\documentclass[12pt]{article}
\usepackage{amssymb}
\usepackage{amsmath}
\usepackage{amsthm}

\newtheorem{thm}{Theorem}[section]
\newtheorem{cor}[thm]{Corollary}
\newtheorem{lemma}{Lemma}
\newtheorem{remark}{Remark}
\newtheorem{prop}[thm]{Proposition}

\def\R{{\mathbb R}}
\def\Z{{\mathbb Z}}
\def\N{{\mathbb N}}

\def\Q{{\mathbb Q}}

\def\1{{\mathbf 1}}
\def\a0{{\aleph_0}}
\def\FF{{\mathcal{F}}}
\def\GG{{\mathcal{G}}}
\def\HH{{\mathcal{H}}}

\def\0{{\mathbf{0}}}

\def\mi{\mu}
\def\ni{\nu}

\def\eps{{\varepsilon}}

\newcommand{\ra}{{\rightarrow}}

\newcommand{\lin}{{\rm lin}}
\newcommand{\supp}{{\rm supp}}

\def\bn{{\big|\big|}}

\title{A series whose sum range is an arbitrary finite set}

\author{Jakub Onufry Wojtaszczyk
\\
Department of Mathematics, Computer Science and Mechanics \\ University of Warsaw\\ ul. Banacha 2,
02-097 Warsaw, Poland\\ email: onufry@duch.mimuw.edu.pl}

\date{March 21, 2004}

\begin{document}

\maketitle
\begin{abstract}
In finitely-dimensional spaces the sum range of a series has to be an affine subspace. It is long
known this is not the case in infinitely dimensional Banach spaces. In particular in 1984 M.I.
Kadets and K. Wo\`{z}niakowski obtained an example of a series the sum range of which consisted of
two points, and asked whether it is possible to obtain more than two, but finitely many points.
This paper answers the question positively, by showing how to obtain an arbitrary finite set as the sum range of a series in any infinitely dimensional Banach
space.
\end{abstract}

\section{Introduction}
For a finitely-dimensional linear space $X$ the well-known Steinitz theorem states that for any
conditionally convergent series the set of all possible limits of the series (called the {\em sum
range}) is a affine subspace of $X$. In the "Scottish Book" S. Banach posed the problem whether
the same holds for infinitely dimensional Banach spaces. The problem was solved negatively in the
same book by J. Marcinkiewicz. In his example the sum range is the set $M$ of all integer-valued functions in $L_2[0,1]$. The next example, due to M. I. Ostrovskii, showed that the sum range does not have
to be a closed set - the sum range of Ostrovskii's series was of the form $M + \sqrt{2}M$.
Finally M. I. Kadets constructed an example in which the sum range consisted of two points,
disproving, in particular, H. Hadwiger's conjecture that the sum range has to be the coset of some
additive subgroup of $X$. The justification of the example was obtained independently by K.
Wo\`{z}niakowski and P. A. Kornilov in 1986.

It is still unknown what sets can be the sum ranges of series. In this paper it is shown that any finite subset of $X$ can be the sum range of a conditionally convergent series,
which solves the problem posed by M. I. Kadets along with his two-point example (the problem is
stated in \cite{Kadets} in the general case, and in \cite{Ulyanov} for $X = C(\Delta)$ and $n=3$).
The example is an extension of the 2-point example of M. I. Kadets as given in \cite{Kadets}. As
far as possible I shall try to keep the notation consistent with the notation given there, although
the lack of suitable letters in the latin alphabet will force me to abandon the notation in a few
places.

Everywhere all spaces are considered with the $L_1$ norm, i.e. $||f||_X = \int_X|f(x)|dx$.
Frequently it is assumed it is obvious on which space the norm is taken, and only $||f||$ is
written.

\section{The results of K. Wo\`{z}niakowski} \label{Wozza}

The work in this paper is strongly inspired by the 2-point example of M. I. Kadets and the proof by
K. Wo\`{z}niakowski. In this paper not only the final result of Wo\`{z}niakowski's work will be
used, but also multiple technical facts than can be found in the proof. Rather than force the
reader to search for those in the original paper, I shall reiterate here Wo\`{z}niakowski's work,
at times formulating the results in a way that will make them easier to use in the subsequent
sections. The whole content of this section in based on \cite{Kadets}, and a reader familiar with
this work may probably skip to the next section.

Let $Q = [0,1]^\omega$ be the infinite dimensional cube, i.e. the product of a countable number of
unit segments, equipped with the standard product topology and measure. By $x = (x_1,x_2,\ldots)$
we shall denote the variable on $Q$. Suppose we have two sequences of functions on the cube:
$a_m^n$ and $b_{m,j}^n$, where $n\in \N$ and for a given $n$ both $m$ and $j$ belong to some finite
sets $M_n$ and $J_n = M_{n+1}$ respectively. By $A_n$ we shall denote the set $\{a_m^n : m \in
M_n\}$, and by $B_n$ the set $\{b_{m,j}^n : m \in M_n, j \in J_n\}$. For convenience if $X$ is a
set of functions, by $\tilde{X}$ we shall denote the sum of the functions from $X$. We shall assume
the following properties of the functions $a_m^n$ and $b_{m,j}^n$:

\begin{eqnarray}
\tilde{A_n}(x) &=& 1 \qquad  \forall_{n\in \N} \forall_{x \in Q} \label{suma1} \\ ||a_m^n|| &=& \frac{1}{|M_n|}
\label{norma} \\ \lim_{n\ra\infty} |M_n| &=&\infty \label{ssize} \\ &&\mbox{The function $a_m^n$
depends only on the variable $x_n$} \label{depa}
\\&& \mbox{The functions $a_m^n$ assume only values 0 and 1} \label{intva} \\
b_{m,j}^n &=& - a_m^n \cdot a_j^{n+1} \label{multab} \end{eqnarray}

We shall this collection of properties the {\em Kadets properties} on the cube $Q$. These properties mean that for each $n$ the interval $[0,1]$ is divided into $|M_n|$ sets $V_m^n$ of equal measure, and $a_m^n (x_1, x_2, \ldots) = 1$ iff $x_n \in V_m$. The functions $b_{m,j}^n$ are negative, and are supported on rectangles $(x_n,x_{n+1}) \in V_m^n \times V_j^{n+1}$.

From the Kadets properties we can easily deduce another few properties, mainly about the behaviour
of $b_{m,j}^n$ based on properties \ref{suma1} and \ref{multab}:

\begin{eqnarray}a_m^n &=& - \sum_{j\in J_n} b_{m,j}^n, \label{absame} \\ a_j^{n+1} &=& -\sum_{m\in M_n}
b_{m,j}^n\label{abdiff} \\ \tilde{B_n}(x) &=& -\1 \forall_{n\in\N} \label{sumb1}
\\ ||b_{m,j}^n|| &=& \frac{1}{|M_n \times J_n|} \label{normb} \\ \\ && \mbox{The function
$b_{m,j}^n$ depends only on the variables $x_n$ and $x_{n+1}$}  \label{depb} \\ && \mbox{The
functions $b_{m,j}^n$ assume only values 0 and -1} \label{intvb} \\ && \mbox{$a_m^n$ and $a_{m'}^n$
have almost disjoint supports for $m \neq m'$} \label{disjoint}\end{eqnarray}

These properties follow easily from the Kadets properties. In property \ref{disjoint} by almost disjoint supports we mean that the intersection of two supports is of measure zero, we can obviously modify $a_m^n$ so that the Kadets properties still hold and the sets $\{x : a_m^n(x) > 0\}$ are disjoint for any constant $n$ and any two different values of $m$.

Let $c_k, k \in \N$ be any ordering of all the functions $a_m^n$ and $b_{m,j}^n$. Following
Wo\`{z}niakowski we shall investigate the convergence of any reordering $c_{\sigma(k)}$ of $c_k$.

\begin{prop} \label{propsums} For any family of functions $c_k$ having the Kadets properties there
exist such two permutations $\sigma$ and $\tau$ of $\N$ that $\sum c_{\sigma(k)} \ra \0$ and $\sum
c_{\tau(k)} \ra \1$.\end{prop}

\begin{proof} For $\sigma$ it is enough to order the functions $a_m^n$ lexicographically, i.e.
$a_m^n$ appears before $a_{m'}^{n'}$ iff $n < n'$ or $n = n'$ and $m < m'$, and then immediately
after each $a_m^n$ to put the whole set $\{b_{m,j}^n : j \in J_n\}$. Then the sum
of each block consisting of a single function $a_m^n$ and the functions $b_{m,j}^n$ following it
sums up to zero due to property \ref{absame}, so the norm of each partial sum is the norm of the
currently open block, which converges to zero due to properties \ref{norma}, \ref{normb} and
\ref{ssize}.

To get $\tau$ we order the functions $a_m^n$ in the same way, but we follow each function $a_m^n$
for $n>1$ by the set $\{b_{l,m}^{n-1} : l \in M_{n-1}\}$, the functions $a_m^1$ are not followed by
anything (as there are no functions $b_{m,j}^0$). Then the functions $a_m^1$ sum up to the constant
function  1 due to property \ref{suma1}. The following blocks again sum up to zero, this time due
to property \ref{abdiff}, so the norm of the difference between 1 and a particular partial sum is
equal to the norm of the currently open block, which again converges to zero due to properties
\ref{norma}, \ref{normb} and \ref{ssize}.\end{proof}

\begin{remark} \label{otnormsums} The series of functions from Proposition \ref{propsums} converge not only in the
$L_1$ norm, but also in any $L_p$ norm for any $p < \infty$.\end{remark}

\begin{proof} Again it is only a question of investigating the norm of any given block, as the sum
of the previous blocks is zero. Functions $a_m^n$ assume only values 0 and 1 and have disjoint
supports for a set $n$ from properties \ref{intva} and \ref{disjoint}. Functions $b_{m,j}^n$ for a
given $n$ have disjoint supports (this follows from properties \ref{multab} and \ref{disjoint}) and
assume values $0$ and $-1$ (from \ref{intvb}). Thus for any $f$ being a sum of any set of functions
$a_m^n$ and $b_{m,j}^n$ for a fixed $n$ (or $a_m^n$ and $b_{m,j}^{n-1}$ for a fixed $n$ in the case of $\tau$) we have $\|f\|_\infty \leq 1$. This implies for any $1 \leq
p < \infty$ $$\|f\|_p = (\int |f|^p)^{1/p} = (\int |f| \cdot |f|^{p-1})^{1/p} \leq (\|f\|_1 \cdot
\|f^{p-1}\|_{\infty})^{1/p} \leq \|f\|_1^{1/p} \cdot 1 = \|f\|_1^{1/p}.$$ Thus if the sum of the
series tended to zero in the $L_1$ norm with $n$ tending to infinity, it also tends to
zero in any $L_p$ norm for $p < \infty$.\end{proof}

\begin{prop} \label{intcons} If a reordering $c_{\sigma(k)}$ of a family $c_k$ having the Kadets
properties converges, it converges to a constant integer function.\end{prop}

\begin{proof} Due to properties \ref{depa} and \ref{depb} and the finiteness of the sets $M_n$ and
$J_n$ only finitely many of the functions $c_{\sigma(k)}$ depend on a given variable $x_l$ -
precisely the functions belonging to $A_l$, $B_l$ and $B_{l-1}$. Moreover the sum of all these
functions equals to the constant function $-1$ due to properties \ref{suma1} and \ref{sumb1}. Thus
for some integer $K_0$ the function $\sum_{k=1}^K c_{\sigma(k)}$ is constant with regard to $x_l$
for $K \geq K_0$, and thus the limit of the series also has to be constant with regard to $x_l$. As
this applied to an arbitrary $l$, the limit simply has to be constant.

As the functions $c_k$ are integer-valued (properties \ref{intva} and \ref{intvb}), their sums also
have to be integer-valued. Thus all the partial sums of the series are integer-valued, and so the
limit is also integer-valued, which ends the proof.\end{proof}

The next step will be to show that 0 and 1 are the only possible limits of a rearrangement of a
family of functions with the Kadets property. We shall set a fixed rearrangement $c_{\sigma(k)}$ of
a given Kadets family, and we shall assume that the sum $\sum_k c_{\sigma(k)}$ converges to some
constant integer $C \neq 1$ (we know $C = 1$ can be achieved, it remains to prove that under these
assumptions $C = 0$).

Take an arbitrary $\delta > 0$ and fix $K_0 = K_0(\delta)$ such that for any $K > K_0$,
\begin{equation} \big{|\big|}C - \sum_{k=1}^Kc_{\sigma(k)}\big{| \big|} \leq \delta
\label{limit}\end{equation} and for any $m > l > K_0$ the Cauchy condition holds, i.e.
\begin{equation} \big{|\big|}\sum_{k=l}^mc_{\sigma(k)}\big{|\big|} \leq \delta. \label{Cauchy}
\end{equation} In addition to the sets $A_n$ and $B_n$ introduced earlier we shall also consider
$V_n = \bigcup_{k=1}^n (A_k \cup B_k)$. Let $M$ be any integer such that \begin{equation}
c_{\sigma(k)} \in V_M \cup A_{M+1} \qquad \mbox{for any $k \leq K$.} \label{MOK} \end{equation} Let
$c_k^* = c_{\sigma(k)}$ if $c_{\sigma(k)} \in V_M \cup A_{M+1}$, 0 otherwise. Similarly let
$\bar{c}_k = c_{\sigma(k)}$ if $c_{\sigma(k)} \in B_{M+1}$, 0 otherwise. By $c^*$ we shall denote
$\sum_{k=K_0+1}^\infty c_k^*$, while by $c$ we shall denote $\sum_{k=1}^{K_0} c_{\sigma(k)}$. The
sum $c + c^*$ is equal to $\tilde{V}_M + \tilde{A}_{M+1} = 0 + 1 = 1$. Hence $||c^*|| = ||1 - c||
\geq ||1 - C|| - ||C - c|| \geq 1 - \delta$. Let $k_0 = K_0$ and \begin{equation} k_{j+1} = \min
\big\{k : \frac{1}{4} - \frac{5\delta}{4} \leq\bn \sum_{i=k_j + 1}^k c_k^* \bn \leq \frac{1}{4} -
\frac{\delta}{4}\big\}.\label{defni}\end{equation}

The indices $k_j$ are well defined for $j$ from 1 to 4 because the total norm of the sum $c^*$ is
at least $1 - \delta$ and each single $c_k^*$ has norm $\leq \delta$ due to the Cauchy condition
(\ref{Cauchy}). For $j = 0,1,2,3$ define the following functions: $$c_{j+1}^{**} = \sum_{k=k_j +
1}^{k_{j+1}} c_k^*, \qquad \bar{\bar{c}}_{j+1} = \sum_{k=k_j + 1}^{k_{j+1}} \bar{c}_k, \qquad
\hat{c}_{j+1} = \sum_{k=k_j + 1}^{k_{j+1}} c_{\sigma(k)},$$ and for $j = 1,2,3,4$ set $r_j =
\hat{c}_j - \bar{c}_j - c^{**}_j$.

In plain words this means that we divide the functions $c_k$ for $k_j < k \leq k_{j+1}$ into three sets - those from $A_n$ for $n \leq M+1$ or $B_n$ for $n \leq M$ (these add up to $c_j^{**}$), those from $B_{M+1}$ (these add up to $\bar{\bar{c}}_j$) and the rest (these add up to $r_j$). We will show that the functions from $B_{M+1}$ are placed in $c_k$ in similar proportions as the functions from $V_M \cup A_{M+1}$ --- if, say, about a half of the functions from $V_M \cup A_{M+1}$ appeared in $c_k$ (that happens at $k_2$) then about a half of the functions from $B_{M+1}$  must have appeared, too.

We shall need to estimate the norm of two sums, which we would like to be negligible: $||r_j||$ and $||\sum_{k=k_4+1}^\infty c_k^*||$. We know that the sum of all $c_k$ up to $k_j$ is negligible, thus if the high-$n$ functions ($r_j$) are negligible, the functions from $V_M \cup A_{M+1}$ and $B_{M+1}$ have to approximately cancel each other out. This motivates the following proposition:

\begin{prop} For a Kadets family of functions $c_k$, its rearrangement $c_{\sigma(k)}$ converging
to some $C \neq 1$, an arbitrary $\delta$ and an arbitrary $M > K_0(\delta)$ as above, with the
notation as above we have $\sum_{j=1}^4 ||r_j|| \leq 18\delta$.\label{rj} \end{prop}

\begin{proof} As $c_j^{**}$ is integer-valued (being a sum of some functions from a Kadets family), the
condition $||c_j^{**}|| \leq \frac{1}{4}$ implies $|\supp c_j^{**}| \leq \frac{1}{4}$. Thus we can use
lemma \ref{l0} (from the section "Auxiliary lemmas")
to get $$||c_j^{**} + r_j|| \geq ||c_j^{**}|| + (1 - 2|\supp c_j^{**}|) ||r_j|| = ||c_j^{**}||
+ \frac{1}{2}r_j.$$ Of course $||\hat{c}_j|| \leq \delta$ from the Cauchy condition (\ref{Cauchy}).
We thus have $$1 \geq \sum_{j=1}^4 ||\bar{\bar{c}}_j|| = \sum_{j=1}^4||\hat{c}_j - c_j^{**} - r_j||
\geq \sum_{j=1}^4||c^{**}_j + r_j|| - \sum_{j=1}^4 ||\hat{c}_j|| \geq $$ $$\geq
\sum_{j=1}^4(||c^{**}_j|| + \frac{1}{2}||r_j||) - 4\delta \geq 1 - 5\delta + \frac{1}{2}\sum_{j=1}^4
||r_j|| - 4\delta,$$ which gives us the sought estimate upon $||r_j||$, namely $\sum_{j=1}^4
||r_j|| \leq 18\delta$. In particular, of course, each $||r_j||$ is bounded by $18\delta$.
\end{proof}

\begin{cor} With the notation and assumptions as above, $||\bar{\bar{c}}_j + c^{**}_j|| \leq
19\delta$ \label{fbg}\end{cor}

\begin{proof} $||\bar{\bar{c}}_j + c^{**}_j|| = ||\hat{c}_j - r_j|| \leq ||\hat{c}_j|| + ||r_j|| \leq \delta + 18\delta = 19\delta$.\end{proof}

\begin{prop} For a Kadets family of functions $c_k$, its rearrangement $c_{\sigma(k)}$ converging
to some $C \neq 1$, an arbitrary $\delta$ and an arbitrary $M > K_0(\delta)$ as above, with the
notation as above we have $||\sum_{k=k_4 + 1}^\infty c_k^*|| \leq 11\delta$. \label{resztowka}
\end{prop}

\begin{proof} We have $$||\bar{\bar{c}}_j|| = ||\hat{c}_j - c_j^{**} - r_j|| \geq ||c_j^{**} + r_j|| - ||\hat{c}_j|| \geq||c_j^{**}|| + \frac{1}{2}||r_j|| -
||\hat{c}_j|| \geq ||c_j^{**}|| - \delta \geq \frac{1}{4} - \frac{9\delta}{4}.$$ Take any index
$k' > k_4$. If the norm $||\sum_{k = k_4 + 1}^{k'} c_k^*||$ were greater then $11\delta$, then there would exist some $k_5 \in (k_4, k']$ such that $12\delta \geq ||\sum_{k = k_4 + 1}^{k_5} c_k^*|| > 11\delta$. Then
by a similar argument ($||\bar{\bar{c}}_5|| \geq ||c_5^{**}|| + (1 - 24\delta)||r_5|| -
||\hat{c}_5|| \geq 11\delta - \delta$) the norm of $\sum_{k=k_4 + 1}^{k_5} \bar{c}_k$ would be larger then $10\delta$ --- but
all the functions $\bar{c}_k$ are negative, so $||\sum \bar{c}_k|| = \sum ||\bar{c}_k||$, which in this case gives $1 \geq ||\sum_{k=k_0}^{k_5} \bar{c}_k|| = \sum_{j=1}^{4} ||\bar{c}_k|| +
||\sum_{k=k_4 + 1}^{k_5}\bar{c}_k|| > 1 - 9\delta + 10\delta$, a contradiction. Thus the norm
$||\sum_{k=k_4 + 1}^\infty c_k^*||$ has to be no greater than $11 \delta$ (the sum is convergent,
as it is in fact the sum of a finite number of functions, all coming from $V_{M+1}$). Let us denote
this sum by $c_5^{**}$.\end{proof}

Now we can prove the main theorem of Wo\`{z}niakowski's work:

\begin{thm} \label{MainWoz} For a Kadets family of functions $c_k$ and some rearrangement
$c_{\sigma(k)}$ converging to $C \neq 1$ we have $|C - \frac{1}{2}| \leq \frac{1}{2}$, which (due
to lemma \ref{intcons}) implies $C = 0$.\end{thm}

\begin{proof} Consider any $\delta$, and the partial sum $S = \sum_{k=1}^{k_4} c_{\sigma(k)}$ with
the notation as above. As $k_4 > K_0$, from assumption \ref{limit} we know that $||S - C|| \leq
\delta$, so it will suffice to estimate $||S - \frac{1}{2}||$. We have $$||S - \frac{1}{2}|| = \bn
c + \sum_{j=1}^4 c_j^{**} + \sum_{j=1}^4 \bar{\bar{c}}_j + \sum_{j=1}^4 r_j + c_5^{**} - c_5^{**} -
\frac{1}{2}\bn =$$ $$= \bn c + c^* - \frac{1}{2} + \sum_{j=1}^{4}\bar{\bar{c}}_j + \sum_{j=1}^4 r_j
- c_5^{**} \bn \leq \bn \frac{1}{2} + \sum_{j=1}^4 \bar{\bar{c}}_j \bn + \bn \sum_{j=1}^4 r_j \bn +
\bn c_5^{**} \bn.$$ The function $\sum_{j=1}^4 \bar{\bar{c}}_j$ is a sum of functions from
$B_{M+1}$, which means assumes only the values 0 and $-1$, thus $|\frac{1}{2} + \sum_{j=1}^4
\bar{\bar{c}}_j|$ is always equal to $\frac{1}{2}$. Inserting this and the bounds upon $r_j$ and
$c_5^{**}$ we get $$||S- \frac{1}{2}|| \leq \frac{1}{2} + 18\delta + 11\delta = \frac{1}{2} +
29\delta.$$ As $||S - C|| \leq \delta$ we get $||C - \frac{1}{2}|| \leq \frac{1}{2} + 30\delta$. As
$\delta$ was chosen arbitrarily, we get the thesis.\end{proof}

\begin{cor} \label{otnorm} The sum range of any Kadets family consists of two points, the constant
functions 0 and 1, in any $L_p$ norm for $1\leq p < \infty$\end{cor}

\begin{proof} From Proposition \ref{propsums} and Remark \ref{otnormsums} we know that the two
constant functions belong to the sum range. From the Proposition \ref{intcons} we know that all
functions in the sum range in the $L_1$ norm are constant integer functions, and from Theorem
\ref{MainWoz} we know that only the two functions $0$ and $1$ are eligible. If any permutation of
the series converged to some function $g$ in some $L_p$ norm, then $\|S_n - g\|_p$ would tend to
zero. But from the H\"{o}lder inequality we know that $\|S_n - g\|_p \geq \|S_n - g\|_1$ (as the
measure of the whole space is 1), which would imply that the series $S_n$ converges also in the
$L_1$ norm, contradicting Theorem \ref{MainWoz}.\end{proof}

\section{The 3-point series}\label{3p}

Denote by $Q_i = [0,1]^\omega, i=1,2,3$ the infinite dimensional cube, i.e., the product of a
countable number of unit segments equipped with the standard product probability measure. The
example will be constructed in $L_1(Q_1 \cup Q_2 \cup Q_3)$. In the whole paper $t =
(t_1,t_2,\ldots)$ will denote the variable on $Q_1$, $u = (u_1,u_2,\ldots)$ will denote the
variable on $Q_2$ and $v = (v_1,v_2,\ldots)$ will denote the variable on $Q_3$.

Our series will consist of functions of three kinds. The functions of the first kind are defined as
follows:
\begin{equation*}f_m^n(t) = \begin{cases}1 & \mbox{if $\frac{m-1}{n} < t_n < \frac{m}{n}$}\\
                                        0 & \mbox{otherwise.}\end{cases}\end{equation*}
\begin{equation*}f_m^n(u) = f_m^n(v) = 0\end{equation*}
for $n \in \N, m \in \{1,2,\ldots,n\}$.

The second kind of functions is defined on all three cubes:

\begin{eqnarray*}g_{m,j}^n(t) &=& \begin{cases} -1 & \mbox{if $\frac{m-1}{n} < t_n < \frac{m}{n}$ and
                                                    $\frac{j-1}{n+1} < t_{n+1} < \frac{j}{n+1}$} \\
                                               0 & \mbox{otherwise} \end{cases}\\
g_{m,j}^n(u) &=& \begin{cases} \frac{1}{n+1} & \mbox{if $\frac{m-1}{n} < u_n < \frac{m}{n}$}\\
                                                0             & \mbox{otherwise}
                                                \end{cases}\\
g_{m,j}^n(v) &=& \begin{cases} 1 & \mbox{if  $\frac{(m-1)(n+1) + j -1}{n(n+1)} < v_n <
\frac{(m-1)(n+1)+j}{n(n+1)}$} \\
                                                0 & \mbox{otherwise} \end{cases}\end{eqnarray*}
for $n \in \N, m \in \{1,2,\ldots,n\}, j \in \{1,2,\ldots,n+1\}$.

The functions of the third kind are defined on $Q_2$ and $Q_3$:
\begin{eqnarray*}
h_{m,j,k}^n(t) &=&0 \\
 h_{m,j.k}^n(u) &=& \begin{cases}-\frac{1}{(n+1)^2(n+2)} &\mbox{if $\frac{m-1}{n} < u_n <
                                                                                        \frac{m}{n}$}\\
                                                0    &\mbox{otherwise}\end{cases}\\
h_{m,j,k}^n(v) &=& \begin{cases} -1 & \mbox{if $\frac{(m-1)(n+1) + j - 1}{n(n+1)} < v_n <
\frac{(m-1)(n+1) + j}{n(n+1)}$ and $\frac{k - 1}{(n+1)(n+2)} < v_{n+1} < \frac{k}{(n+1)(n+2)}$}\\
    0 & \mbox{otherwise}\end{cases}\end{eqnarray*}
for $n \in \N, m \in \{1,2,\ldots,n\}, j \in \{1,2,\ldots,n+1\}, k \in \{1,2,\ldots, (n+1)(n+2)\}$.

These functions have properties we want to generalize. Suppose we have three families of indices:
$M_n$, $J_n$ and $K_n$, with $J_n = M_{n+1}$ and $K_n = M_{n+1} \times J_{n+1}$ (here $M_n =
\{1,2,\ldots,n\}$ and the mapping between $\{1,2,\ldots,n\} \times \{1,2,\ldots,n+1\}$ and
$\{1,2,\ldots,n(n+1)\}$ is given by $(m,j) \mapsto (m-1)(n+1) + j$). We have three families of
functions: the first kind $\{f_m^n : n\in \N, m\in M_n\}$, the second kind $\{g_{m,j}^n : n\in\N,
m\in M_n, j\in J_n\}$ and the third kind $\{h_{m,j,k}^n : n\in\N, m\in M_n, j\in J_n, k \in K_n\}$
defined on the union $Q_1 \cup Q_2 \cup Q_3$ of Hilbert cubes. The families $f$ and $g$ form a
Kadets family on $Q_1$, while the functions $h$ disappear on $Q_1$. On $Q_3$ the functions $g$ and
$h$ form a Kadets family (with $M_n \times J_n$ being the first index set and $K_n$ the second),
while functions $f$ disappear. The properties of the functions on $Q_2$ are different, as follows:

\begin{eqnarray}
 \sum_{m \in M_n} \sum_{j \in J_n} g_{m,j}^n &=& \1 \label{gconst2}\\
 \sum_{m \in M_n} \sum_{j \in J_n} \sum_{k \in K_n} h_{m,j,k}^n &=& -\1,\label{hconst2}\\
 g_{m,j}^n &=& -\sum_{k \in K_n} h_{m,j,k}^n, \label{hgsingle}\\
\sum_{m' \in M_{n+1}} g_{m',j'}^{n+1} &=& - \sum_{m\in M_n} \sum_{j\in J_n} \sum_{m' \in M_{n+1}}
h_{m,j,(m',j')}^n.\label{hglevel}\\ \sum_{j\in J_n} g_{m,j} && \mbox{assumes only values 0 and 1}
\label{gaddup} \\ \int_{Q_2} g_{m,j}^n &=& \int_{Q_3} g_{m,j}^n \label{gint} \\ \int_{Q_2}
h_{m,j,k}^n &=& \int_{Q_3} h_{m,j,k}^n \label{hint} \\ ||g_{m,j}^n|| &=& \frac{1}{|M_n \times J_n|}
\label{normg} \\ ||h_{m,j,k}^n|| &=& \frac{1}{|M_n \times J_n \times K_n|} \label{normh}
\\ &&\mbox{The functions $g_{m,j}^n$ and $h_{m,j,k}^n$ on $Q_2$ depend only on $u_n$} \label{dep2}
\end{eqnarray}

Such a family of functions will be called a {\em 3-Kadets family}. It is easy (although maybe a bit
tedious) to check that the family defined at the beginning of the section is a 3-Kadets family.

We shall denote by $F_n$ the set $\{f_m^n : m \in M_n \}$, by $G_n$ the set $\{g_{m,j}^n : m \in
M_n; j \in J_n \}$ and by $H_n$ the set $\{h_{m,j,k}^n : m \in M_n, j \in J_n; k\in K_n\}$. Also, by
$V_M$ we shall denote $\bigcup_{k=1}^M F_k \cup G_k \cup H_k$. Denote by $d_n$ any set enumeration
of the whole 3-Kadets family. We are investigating the possible limits of $\sum_{n=1}^\infty
d_{\sigma(n)}$ for all permutations $\sigma$ of $\N$.

If a given rearrangement $d_{\sigma(n)}$ of a 3-Kadets family converges, it converges on each of
the cubes separately. On $Q_1$ and $Q_3$ we have Kadets families of functions, so the series on
each of these cubes converges either to $\0$ or to $\1$ due to theorem \ref{MainWoz}. The new part
is the behaviour on $Q_2$. Same as in the first part of Proposition \ref{intcons} only finitely
many functions depend on a given variable $u_n$ -- the functions $g_{m,j}^n$ and $h_{m,j,k}^n$ --
and their sum is constant, equal to zero due to property (\ref{hgsingle}) applied to each $j$
separately. Thus again the limit of the series $\sum d_{\sigma(n)}$ on $Q_2$ has to be a constant
function.

As $\int_{Q_2} d_n = \int_{Q_3} d_n$ for any $d_n$ (it is 0 for functions of the first kind and
follows from properties \ref{gint} and \ref{hint} for the second and third kind), we get
$\int_{Q_2} \sum_{n=1}^N d_{\sigma(n)} = \int_{Q_3} \sum_{n=1}^N d_{\sigma(n)}$. As the integral is
a continuous functional on $L_1(Q_2)$ and $L_1(Q_3)$ we get that the integrals of the limits have
to be equal -- but we know that the limit of $\sum d_{\sigma(n)}$ on both $Q_2$ and $Q_3$ is a
constant function, so the equality of integrals implies the equality of the limits. Thus the limit
of the whole series is described by a pair of integers - the value on $Q_1$ and the value on $Q_3$.
Let us denote the limit function by $d_\infty$.

We are to show that it is possible to obtain exactly three different sums -- precisely we can
obtain $(\0,\0), (\1,\0)$ and $(\1,\1)$. To obtain any of these limits we first arrange the
functions $f$ and $g$ as by Proposition \ref{propsums} for a Kadets family on $Q_1$, and then after
each $g$ we put the $h$ functions as by Proposition \ref{propsums} for the cube $Q_3$. It remains
to be seen if we get convergence on $Q_2$.

In the case of $(\0,\0)$ after a given $f_m^n$ there appear the all functions $g_{m,j}^n$ and
$h_{m,j,k}^n$ with the same $m$ and $n$. The sum of all these functions on $Q_2$ is equal to $\0$
due to property (\ref{hgsingle}) for each $j$ separately. Thus the norm of the partial sum on $Q_2$
is equal to the norm of the functions appearing after the last $f$, and this tends to zero due to
properties \ref{normg}, \ref{normh} and \ref{ssize} (all the functions have the same index $m$, so
the sum of their norms is equal to $\frac{2}{|M_n|} \ra 0$).

In the case of $(\1,\0)$ after a given $f_m^n$ we get the functions $g_{l,m}^{n-1}$ and
$h_{l,m,k}^{n-1}$. The sum of all these functions on $Q_2$ is again $\0$ due to property
\ref{hgsingle}, this time applied to each $l$ separately. Again the norm of the difference between
the  partial sum and $(\1,\0)$ is the norm of the part after the last $f$, and that again tends to
0.

In the case of $(\1,\1)$ after a given $f_m^n$ we get the functions $g_{l,m}^{n-1}$ and
$h_{l',m',(l,m)}^{n-2}$. Their sum is $\0$ due to property \ref{hglevel} applied to them all. Again
the norm of the difference between the partial sum and $\1$ tends to 0.

Again it is easy to check that the convergence occurs not only in the $L_1$ norm, but also in any
$L_p$ norm for $p < \infty$ in the same way as in Remark \ref{otnormsums} --- on each of the cubes
the $L_\infty$ norm of the partial sums is bounded by 1.

One may wonder why the same arguments will not imply the convergence of the series arranged by rows
in $G_n$ and columns in $H_{n-1}$ to $(\0,\1)$. The answer is we lack the equivalent of property
\ref{hglevel} for this arrangement. To illustrate this let us look at the 3-Kadets family given at
the beginning of the section arranged in this natural way. The sum $\sum_{j=1}^{n+1}g_{m,j}^n$ on
$Q_2$ is equal to $1$ on $\frac{m-1}{n} < u_n < \frac{m}{n}$, while the sum of the appropriate
column of $H_{n-1}$, $\sum_{j=1}^{n+1}\sum_{m'=1}^{n-1}\sum_{j'=1}^nh_{m',j',(m-1)(n+1) + j}^{n-1}$
is equal to $-\frac{1}{n}$ on the whole cube $Q_2$. Thus the partial sums before each function of
the first kind do not disappear as they did in the previous three cases, and when half of these
functions from a given $F_n$ have appeared, the norm of the partial sum on $Q_2$ is $\frac{1}{2}$
regardless of $n$ -- thus this particular series does not converge. Of course we still have to
prove this is true for any rearrangement -- but this example shows the nature of the reason why
only three and not four possible limits exist.

\section{Auxiliary lemmas}

Before we begin the main part of this paper -- i.e. the proof that our series cannot converge to
$(\0,\1)$ -- we shall need three auxiliary lemmas:

\begin{lemma} \label{l0} (Lemma given without proof in \cite{Wozniakowski}) Let $(X,\mi)$ and
$(Y,\ni)$ be measure spaces with probability measures. Let $f(x,y)$ and $g(x,y)$ be functions in
$L_1(X\times Y)$, each of which depends on only one variable: $f(x,y) = \tilde{f}(x), g(x,y) =
\tilde{g}(y)$. Then $$||f+g|| \geq ||f|| + ||g||[1 - 2\mi(\supp \tilde{f})].$$\end{lemma}

\begin{proof} $||f+g|| = \int_{X\times Y} |f+g| = \int_{\supp\tilde{f} \times Y} |f+g| + \int_{(X
\setminus \supp \tilde{f}) \times Y} |g| \geq \int_{\supp(\tilde{f})\times Y} |f| -
\int_{\supp(\tilde{f})\times Y} |g| + (1-\mi(\supp\tilde{f})) ||g|| = ||f|| - \mi(\supp
\tilde{f})||g|| +(1-\mi(\supp\tilde{f})) ||g|| = ||f|| + ||g||[1 - 2\mi(\supp
\tilde{f})].$\end{proof}

\begin{lemma} \label{l1} Let $A,B,C$ be arbitrary spaces equipped with probabilistic measures and let $X = A\times
B\times C$ be equipped with the standard product measure. Suppose $f,g$ are bounded functions defined
on $X$ of the form $f(a,b,c) = \tilde{f}(a,b) = \sum_{k=1}^N s_k \chi_{A_k \times B_k}$ and $g(a,b,c) = \tilde{g}(b,c) = \sum_{l=1}^N t_l \chi_{B_l \times C_l}$, and $\|f - g\| \leq \eps$. Then there exists a function $h(a,b,c) = \tilde{h}(b)$ such
that $\|h - g\| \leq 2\eps$ and $\|h - f\| \leq 2\eps$. Moreover if $f$ is integer-valued then $h$
can also be chosen to be integer-valued, and if for a family of sets $B_\alpha$ we have
$\forall_\alpha \forall_{b_1, b_2 \in B_\alpha} \forall_{a\in A} f(a,b_1,c) = f(a,b_2,c)$, then we
can choose a function $h$ constant on any set $B_\alpha$.\end{lemma}

\begin{proof} For any given $b \in B$ we take $\tilde{h}(b)$ such that $$\int_A |\tilde{f}(a,b) -
\tilde{h}(b)| da = \inf_{x \in \R} \{\int_A |\tilde{f}(a,b) -x| da\}.$$ This is well defined, as
$f$ is bounded, and thus in fact the $\inf$ is taken over a bounded, and thus compact set. For such
an $h$ we have $$\|h - f\| = \int_X |f(a,b,c) - \tilde{h}(b)| = \int_C \int_B \int_A
|\tilde{f}(a,b) - \tilde{h}(b)| = \int_C \int_B \inf \{ \int_A |\tilde{f}(a,b) - x(b)|\} \leq $$ $$\leq \int_C \int_B
\int_A |\tilde{f}(a,b) - \tilde{g}(b,c)| \leq \int_C \int_B \int_A
|f(a,b,c) - g(a,b,c)| = \|f-g\| \leq \eps.$$ As $\|h-f\| \leq \eps$ and $\|f-g\| \leq \eps$, we
immediately have $\|g-h\| \leq 2\eps$. As for the additional assumptions, if $f$ and $g$ are
integer-valued, we can take the $\inf$ in the definition of $\tilde{h}$ to be taken only over
integers, with the same result. Regardless of that which option we choose, if $f$ is constant with
regard to $b$ on any $B_\alpha$, then from the definition $h$ also can be chosen to be constant on
that set.\end{proof}

\begin{lemma} \label{l2} Let $A,B$ be arbitrary spaces equipped with probabilistic measures and $X = A\times
B$ equipped with the standard product measure. Suppose $f,g,h$ are integer-valued functions defined
on $X$ fulfilling $f(a,b) = \tilde{f}(a)$ and $h(a,b)=\tilde{h}(b)$ for some $\tilde{f},
\tilde{h}$. Suppose too that the function $g$ assumes only two adjacent values (i.e. $k$ and $k+1$
for some $k$) . Finally suppose that $\|f + g + h\| < \delta < \frac{1}{9}$. Then either $f$ or $h$
is a constant function equal some integer $c$ on a set of measure $\geq 1 - 2\sqrt{\delta}$.
Furthermore the function satisfies $\|f-c\| < 3\sqrt{\delta}$ (or $\|h-c\| < 3\sqrt{\delta}$,
respectively).
\end{lemma}

\begin{proof} The sets $F_n = \tilde{f}^{-1}((-\infty,n])$ and $H_n = \tilde{h}^{-1}((-\infty,n])$
form two increasing families, the sum of each is the whole space $X$ and the intersection of each
is empty. The measures $|F_n|$ thus form an ascending sequence with elements arbitrarily close to 0
when $n \ra -\infty$ and arbitrarily close to 1 when $n\ra \infty$. As $F_n \setminus F_{n-1} =
\tilde{f}^{-1}(n)$, if $\tilde{f}$ is not constant on any set of measure $\geq 1 - 2\sqrt{\delta}$,
then at least one element of the sequence $|F_n|$, say $F_{n_f}$, has to fall into the interval
$[\sqrt{\delta}, 1-\sqrt{\delta}]$. Similarly if $\tilde{h}$ is constant on no set of measure $\geq
1 - 2\sqrt{\delta}$, then for some $n_h$ we have $\sqrt{\delta} \geq |H_{n_h}| \geq 1 -
\sqrt{\delta}$. Then on the set $X_1 = F_{n_f} \times H_{n_h}$ we have $f(a,b) + h(a,b) \leq n_h +
n_f$, while on $X_2 = (A \setminus F_{n_f}) \times (B \setminus H_{n_h})$ we have $f(a,b) + h(a,b)
\geq n_h + n_f + 2$. As $g$ assumes two adjacent values, it is either $ \leq -(n_h + n_f + 1)$ or
$\geq -(n_h + n_f + 1)$ on the whole space $X$. Thus on one of the sets $X_1, X_2$ we have $|f + g
+ h| \geq 1$, call it $X_i$. As both $X_1$ and $X_2$ are products of two sets of measure $\geq
\sqrt\delta$, we have $\|f + g + h\| = \int_X |f(a,b) + g(a,b) + h(a,b)| \geq \int_{X_i} |f(a,b) +
g(a,b) + h(a,b)| \geq |X_i| \geq \delta$, which contradicts the assumptions of the lemma.

Thus one of the functions has to be constant on a large set. Without the loss of generality we may
assume it is $h$, and that it is equal to some integer $c$. Let us examine the function $f$, taking
into account that all the functions are integer-valued, and thus if their sum is non-zero, it is at
least one : $$\delta > \|f + g+h\| \geq \|f+g+c\|_{A\times h^{-1}(c)} \geq |\{\tilde{f}(a) \not \in
\{-k-c, -k-c-1\}\} \times h^{-1}(c)| =$$ $$=|\{\tilde{f}(a) \not \in \{-k-c, -k-c-1\}\}| \cdot
(1-2\sqrt{\delta}),$$ which implies $\tilde{f}(a) \in \{-k-c,-k-c-1\}$ on a set of measure at least
$1 - \frac{\delta}{1-2\sqrt{\delta}}$. Denote this set by $A'$. Now we return to the function $h$:
$$\|h - c\|_X \leq \frac{1}{1-2\sqrt{\delta}}\|h - c\|_{A'\times B} = \frac{1}{1 -
2\sqrt{\delta}}\|h-c\|_{A'\times (B \setminus h^{-1}(c))}.$$ On the set $A'$ the function $f+g + c$
assumes values of absolute value $\leq 1$, so by substituting $f+g$ for $-c$ we shall decrease the
norm at most by $$1 \cdot |A' \times (B\setminus h^{-1}(c))| \leq (1 -
\frac{\delta}{1-2\sqrt{\delta}})(2\sqrt{\delta}) \leq 2\sqrt{\delta},$$ thus giving the inequality
$$\|h-c\|_X \leq \frac{1}{1-2\sqrt{\delta}}\|h+f+g\|_{A' \times (B\setminus h^{-1}(c))} +
2\sqrt{\delta} \leq \frac{1}{1-2\sqrt{\delta}}\|f+g+h\|_X + 2\sqrt{\delta} \leq
\frac{\delta}{1-2\sqrt{\delta}} + 2\sqrt{\delta}.$$ As $\delta \leq \frac{1}{9}$, we have
$\frac{\delta}{1-2\sqrt{\delta}} \leq \sqrt{\delta}$, and thus $||h-c|| \leq 3\sqrt{\delta}$.

\end{proof}

\section{The fourth point}\label{4p}

Now we can begin to prove the main theorem of the paper:

\begin{thm} The function $d_\infty = (\0,\1)$ does not belong to the sum range of any 3-Kadets family
series.\label{Main}\end{thm}

\begin{proof} Suppose we have a rearrangement of some 3-Kadets family $d_{\sigma(n)}$ the sum of which
converges to $d_\infty$. Again, take an arbitrarily small $\delta
> 0$ (we shall need $927\sqrt\delta < \frac{1}{4}$, i.e. $\delta < \frac{1}{13749264}$) and an
integer $K$ satisfying inequalities (\ref{limit}) and (\ref{Cauchy}), i.e. the tails and
Cauchy sums are smaller than $\delta$ for $N>K$. Then, again, we take any $M$ satisfying
(\ref{MOK}), i.e. such that $V_M$ contains the first $K$ elements of our series. Then we take an
$N_0$ such that
\begin{equation}V_M \subset \{d_{\sigma(1)}, d_{\sigma(2)}, \ldots,
d_{\sigma(N_0)}\}.\label{N0OK}\end{equation} Consider any fixed $N > N_0$. We will prove that
\begin{equation*}\int_{Q_3} \sum_{n=1}^N d_{\sigma(n)}
< \frac{1}{4}.\end{equation*} Of course this suffices to prove that our series does not converge to
$1$ on $Q_3$, which contradicts the assumption the rearrangement converged to $(\0,\1)$.

Denote for any $L,k\in \Z$ by $D_k$ the set $\{d_{\sigma(1)},\ldots,d_{\sigma(k)}\}$, and by
$F^k_L, G^k_L, H^k_L$ and $V^k_L$ the intersections of sets $F_L, G_L, H_L$ or $V_L$, respectively,
with the set $D_k$. First we shall prove the following lemma:

\begin{lemma} \label{l3} If functions $f^n_m, g_{m,j}^n$ and $h_{m,j,k}^n$ are a 3-Kadets family
on the cubes $Q_1, Q_2$ and $Q_3$ and their set permutation $d_{\sigma(n)}$ tends to $0$ on $Q_1$
and $1$ on $Q_2$ and $Q_3$, and for a given $L$ we have $\int_{Q_3} \tilde{G}^N_L \geq \frac{1}{2}
+ 38\delta$, where $N > N_0$ as above, then there exists a $P \subset [0,1]$ such that $|P| =
\frac{1}{2}$ and $[(\tilde{H}_L^N)^{-1}(0)] \cap \{v: v_L \in P\} \subset Q_3$ has measure $\leq
450\delta$.\end{lemma}

\begin{remark} What this lemma really tells us is: if up to the $N$th element of the series at
least half plus something $(38\delta)$ of the $G_L$ functions have appeared, then at least half
minus something $(450\delta)$ of the $H_L$ functions had to appear. Moreover, the $H_L$ functions
do not appear in a haphazard fashion - we know that at least half minus something rows had to
appear (a row is the set of the functions $h_{m,j,k}^L$ with fixed $m$ and $j$ and varying
$k$).\end{remark}

\begin{proof} If $L \leq M$ then our thesis is automatically fulfilled -- all functions from $H_L$ belong
to the set $D_N$, thus we can take any set of measure $\frac{1}{2}$ for $P$ and the set
$(\tilde{H}_L^N)^{-1}(0)$ will be empty, so $P$ will satisfy the required conditions.

Now consider the case $L > M$. The numbers $K$ and $L-1$ satisfy the conditions (\ref{limit}),
(\ref{Cauchy}) and (\ref{MOK}) (as $L>M$ and $M$ satisfied (\ref{MOK})). Thus we know there exist
numbers $n_i$ satisfying (\ref{defni}). We shall prove that $N \geq n_2$. 

We know that $\int_{Q_3} \tilde{G}^N_L = -\int_{Q_1} \tilde{G}^N_L$ (as all $g_{m,j}^n$ are of the
same constant sign on each cube, the absolute value of the integral is equal to the norm, and the
norms on each cube are equal) . If $N < n_2$, then $$\|\tilde{G}_L^N\|_{Q_1} \leq
\|\tilde{G}_L^{n_2}\|_{Q_1} = \|\bar{\bar{d}}_1 + \bar{\bar{d}}_2\| \leq \|d^{**}_1\| + 19\delta +
\|d^{**}_2\| + 19\delta < \frac{1}{2} + 38\delta,$$ which contradicts our assumption (the first
inequality follows from the fact, that $g_{m,j}^n$ are all non-positive functions on $Q_1$, the
second inequality from corollary \ref{fbg}).

Thus $N > n_2$. Consider $\tilde{V}_{L-1}^{n_2} + \tilde{F}_L^{n_2}$ on $Q_1$. This function is
dependent on variables $t_1,t_2,\ldots,t_L$, while $\tilde{G}_L^{n_2} = \bar{\bar{d}}_{1} +
\bar{\bar{d}}_{2}$ on $Q_1$ depends on $t_L$ and $t_{L+1}$. From property (\ref{limit}) and
Corollary \ref{fbg} we get
\begin{equation*} \|\tilde{V}_{L-1}^{n_2} + \tilde{F}_L^{n_2} + \tilde{G}_L^{n_2}\| \leq
\| \tilde{D}_k \| + \| d^{**}_{1} + \bar{\bar{d}}_{1}\| + \| d^{**}_{2} + \bar{\bar{d}}_{2}\| \leq
\delta + 19\delta + 19\delta = 39\delta.\end{equation*}

We can thus use lemma \ref{l1} for functions $-\tilde{V}_{L-1}^{n_2} - \tilde{F}_L^{n_2}$ and
$\tilde{G}_L^{n_2}$ to get that on $Q_1$ both these functions are closer than $39\delta$ to some
integer-valued function $\tilde{A}$ depending only on $t_L$.

Each function $f_m^n$ depends only on $t_n$ and assumes values $0$ and $1$ only (properties
\ref{intva} and \ref{depa}), so it is in fact the characteristic function of a set $\{t : t_n \in
S_m^n\}$ for some $S_m^n \subset [0,1]$. As the $f_m^n$ functions have disjoint support for a fixed
$n$, they are all constant on any given $S_m^n$. The $g$ functions are also constant with regard to
$t_n$ on the $S_m^n$ due to property \ref{multab}, and all the other functions are constant with
regard to $t_n$ on the whole interval. Thus the functions $-\tilde{V}_{L-1}^{n_2} -
\tilde{F}_L^{n_2}$ and $\tilde{G}_L^{n_2}$ are constant with respect to $t_L$ on sets $\{t_L \in
S_m^L\}$ we can choose $\tilde{A}$ to be constant on those sets. Thus $\tilde{A}$ coincides on $Q_1$ with the sum of some of the rows of $G_L$, i.e. $\tilde{A}$ corresponds to some subset $A$ of $G_L$ such that for a fixed $m$ either all or none of the functions $g_{m,j}^L$ belong to $A$. Define $\tilde{A}$ on $Q_2$ and $Q_3$ as the sum of all the elements of $A$ as well, which agrees with our notation that $\tilde{U}$ is the sum of all the elements of $U$ for an arbitrary set of functions.

We know from (\ref{defni}) and Proposition \ref{resztowka} that $\|\sum_{n=n_2+1}^\infty d_n^*\|_{Q_1} \leq \frac{1-\delta}{4} + \frac{1-\delta}{4} + 11\delta \leq \frac{1}{2} + 11\delta$. Remark that $(\tilde{V}^{n_2}_{L-1} + \tilde{F}_L^{n_2} + \sum_{n=n_2 + 1}^\infty d^*_n)|_{Q_1} = (\tilde{V}_{L-1} + F_L)|_{Q_1} = \1|_{Q_1}$, so $\|\tilde{V}_{L-1}^{n_2} + \tilde{F}_L^{n_2}\|_{Q_1} \geq \frac{1}{2} - 11\delta$. On the other hand $\|\tilde{V}_{L-1}^{n_2} + \tilde{F}_L^{n_2}\|_{Q_1} = \|\tilde{D}_K + d_1^{**} + d^{**}_2\|_{Q_1} \leq \delta _ \frac{1-\delta}{4} + \frac{1 - \delta}{4} \leq \frac{1}{2} + \delta$. As $\|\tilde{V}_{L-1}^{n_2} + \tilde{F}_L^{n_2} - \tilde{A}\|_{Q_1} \leq 39\delta$, taking into account the equality $\|\tilde{A}\|_{Q_1} = \|\tilde{A}\|_{Q_2}$ we can estimate that

\begin{equation}\frac{1}{2} - 50\delta \leq \|\tilde{A}\|_{Q_2} \leq \frac{1}{2} +
40\delta.\label{measat}\end{equation}

Distinct functions from $G_L$ have disjoint supports on $Q_1$ (this follows from the properties
\ref{disjoint} and \ref{multab} of Kadets families), and each has the same norm $\psi =
\frac{1}{|M_L \times J_L|}$. Thus if the distance between two functions corresponding to two
subsets of $G_L$ on $Q_1$ is smaller than $n\psi$, then at most $n$ functions belong to the
symmetric difference of those two subsets. If at most $n$ functions belong to the symmetric
difference, then the distance between the two functions on $Q_2$ is at most $n\psi$ (as on $Q_2$
the norm of a single function is also equal $\psi$ by property \ref{normg}). Thus, in general, if
$B,C\subset G_L$, then $\|\tilde{B} - \tilde{C}\|_{Q_1} \geq \|\tilde{B} - \tilde{C}\|_{Q_2}$. In
particular $\tilde{G}_L^{n_2}$ is at most $39\delta$ distant from $\tilde{A}$ on $Q_2$.

Now consider what happens on $Q_2$. From (\ref{gaddup}) the
restriction of $\tilde{A}$ to $Q_2$ is equal to $1$ on some set
(on intervals $t_L \in [\frac{m-1}{L},\frac{m}{L}]$ for $m$ such
that $g^L_{m,j} \in A$) and 0 on the rest. From (\ref{limit}), as
$n_2 > K$, we have $\|\tilde{D}_{n_2} - 1\|_{Q_2} \leq \delta$. If
we substitute $\tilde{A}$ for $\tilde{G}_L^{n_2}$, we will be at
most $40\delta$ distant from zero, precisely
\begin{equation*} \|\tilde{D}_{n_2} - 1 - \tilde{G}^{n_2}_L + \tilde{A}\|_{Q_2} \leq
40\delta.\end{equation*} However as only $G_L$ and $H_L$ depend on $u_L$, this sum is composed of
two parts - the part $\tilde{A} + \tilde{H}_L^{n_2}$ dependent on $u_L$ and the whole rest (i.e.
$\tilde{D}_{n_2} - (\tilde{G}_L^{n_2} + \tilde{H}_L^{n_2})$) dependent on other variables. Thus we
can apply a simplified version of lemma \ref{l1}, with $f = \tilde{A} + \tilde{H}_L^{n_2}$, $g=
-(\tilde{D}_{n_2} - V_L^{n_2})$, and a trivial one-point space as $B$. We learn that both our
functions are within $80\delta$ from a function $c$ dependent on $b$ -- but as $B$ was a one-point
space, $c$ is a constant function. As $\tilde{A}$ assumes values $0$ and $1$, and
$\tilde{H}_L^{n_2} \in [-1,0]$, their sum is non-negative on $\supp\tilde{A}$ and non-positive on
the remainder of $Q_2$.

From (\ref{measat}) we know that $|\supp\tilde{A}| \geq \frac{1}{2} - 50\delta$, thus $\tilde{A} +
\tilde{H}_L^{n_2}$ is non-negative on a set of measure $\geq \frac{1}{2} - 50\delta$. If $c$ is
positive, then (as $\delta < \frac{1}{200}$)
\begin{equation*}80\delta \geq \|\tilde{A} + \tilde{H}_L^{n_2} - c\| \geq c (\frac{1}{2} - 50\delta) \geq
\frac{c}{4},\end{equation*} which implies $c \leq 320\delta$.
Similarly if $c$ is negative, we know from (\ref{measat}) that
$|Q_2 \setminus \supp\tilde{A}| \geq \frac{1}{2} - 40\delta$,
yielding again $c > -\frac{800}{3}\delta$. Thus $|c| < 320\delta$,
so $\|\tilde{A} + \tilde{H}_L^{n_2}\| \leq \|\tilde{A} +
\tilde{H}_L^{n_2} - c\| + |c| \leq 80\delta + 320\delta =
400\delta$.

Thus $\tilde{H}_L^{n_2}$ is within $400\delta$ of a function with values 0 and -1 on $Q_2$ -- the
function $-\tilde{A}$. Remark, that $-\tilde{A} = -\tilde{A}'$ on $Q_2$ for a subset $A'$ of $H_L$ with the property that for a given $m$ either all of the functions $h_{m,j,k}^L$ belong to $A'$, or none of
the functions belongs to $A'$ (if a given $g_{m,j}^L$ belongs to $A$, then all $h_{m,j,k}^L$ belong
to $A'$) . If $\tilde{A'}$, where $A' \subset H_L$, is a function assuming only values $0$ and $1$
on $Q_2$ and $B \subset H_L$, then
\begin{eqnarray*}\|\tilde{A'} - \tilde{B}\|_{Q_2} &=& \|\tilde{A'} - \tilde{B}\|_{\supp \tilde{A'}}
+ \| \tilde{A'} - \tilde{B}\|_{\Q_2 \setminus \supp\tilde{A'}} \\ &=& \frac{1}{|M_L \times J_L
\times K_L|}|\{h : h \in A' \wedge h \not\in B\}| + \frac{1}{|M_L \times J_L \times K_L|}|\{h :
h\not\in A' \wedge h\in B\}|
\\ &=& \frac{1}{|M_L \times J_L
\times K_L|}|A \bigtriangleup B| = \|\tilde{A'} -\tilde{B}|_{Q_3}.\end{eqnarray*}

Let us take any subset $A''$ of $H_L$ depending only on $m$ and $j$ with exactly half of the
elements of $H_L$ and containing $A'$ or contained in $A'$. If $B \subset C \subset H_L$ or $C
\subset B \subset H_L$, then $\|\tilde{C} - \tilde{B}\| = |\|\tilde{C}\| - \|\tilde{B}\||$, because
all the the functions in $H_L$ are non-positive. As $\tilde{A'} = - \tilde{A}$ on $Q_2$ and from
(\ref{measat}) $|\|\tilde{A}\|_{Q_2} - \frac{1}{2}| \leq 50\delta$, we get $\|\tilde{A'} - \tilde{A''}\|_{Q_2}
\leq 50\delta$, and thus $\|H_L^{n_2} -\tilde{A''}\|_{Q_3} = \|H_L^{n_2} - \tilde{A''}\|_{Q_2} \leq 450\delta$.

Now consider what happens on $Q_3$. As $\tilde{H}_L^{n_2}$ and
$\tilde{A''}$ are both integer-valued on $Q_3$, this means they
differ on a set of measure at most $450\delta$, and thus their
difference can be positive on a set of measure at most
$450\delta$. When we increase $n$ from $n_2$ to $N$ the set where
the difference is positive can only decrease. Thus $|\{H_L^N -
\tilde{A''} > 0\}| \leq 450\delta$. Now for $P$ we take $\supp
\tilde{A}''$. The set $[(\tilde{H}_N^L)^{-1}(0)] \cap \{v: v_L \in
P\}$ is the set where $H_N^L$ is equal to zero and $\tilde{A}''$
is negative --- thus their difference is positive, so the set has
to have measure smaller than $450\delta$, which is what we had to
prove.
\end{proof}

Now the main proof. Assume $d_\infty = (\0,\1)$, i.e. our series converges to $\1$ on $Q_2$ and
$Q_3$ and to $\0$ on $Q_1$. We shall prove by induction upon $L$ that $\int_{Q_3} \tilde{V}_L^N
\leq \frac{1}{4}$. As $\sum_{n=1}^N d_{\sigma(n)}$ is finite, its elements are contained in some
$V_L$, thus if the thesis is true, we get $\int_{Q_3}\sum_{n=1}^N d_{\sigma(n)} \leq \frac{1}{4}$,
which is what we had to prove. For $L < M$ we have $V_L \subset D_N$ and from property
(\ref{absame}) $\int_{Q_3} \tilde{V}_L^N = 0 \leq \frac{1}{4}$. Now suppose we have the thesis for
$L-1$ and attempt to prove it for $L$. Denote by $P_1$ the function $(\tilde{V}_{L-1}^N +
\tilde{G}_L^N) |_{Q_3}$ and by $P_2$ the function $\sum_{n > L} \tilde{G}_n^N +
\tilde{H}_n^N|_{Q_3}$. Consider the function $\tilde{H}_L^N|_{Q_3}$. It depends on variables $v_L$
and $v_{L+1}$. The function $P_1$ depends on $v_1,\ldots, v_L$, while $P_2$ depends on
$v_{L+1},\ldots,v_Z$ for some $Z\in \Z$. The function $H_L^N|_{Q_3}$ assumes only values 0 and -1,
all three functions -- $H_L^N|_{Q_3}$, $P_1$ and $P_2$ are integer-valued, and from (\ref{limit})
their sum is less then $\delta$ distant from $\1$ on $Q_3$. Thus by taking $P_1' = P_1 - 1$ we have
three functions fulfilling the assumptions of lemma \ref{l2}. Thus either $P_1$ or $P_2$ is within
$3\sqrt{\delta}$ of a constant function. In each of these cases the proof will also depend on
whether $\int_{Q_3} \tilde{G}_L^N \leq \frac{1}{2} + 38\delta$ or $\int_{Q_3} \tilde{G}_L^N >
\frac{1}{2} + 38\delta$. Thus we have in total four cases to consider.

Suppose first that $P_2$ is within $3\sqrt\delta$ of a constant function. As $\|P_1 + P_2 +
\tilde{H}_L^N - 1\| \leq \delta$, this means that $P_1 + \tilde{H}_L^N$ is within $3\sqrt\delta +
\delta \leq 4\sqrt\delta$ of a constant function. If $\int_{Q_3} \tilde{G}_L^N \leq \frac{1}{2} +
38\delta$, then $\int_{Q_3} \tilde{V}_L^N = \int_{Q_3} \tilde{V}_{L-1}^N + \tilde{G}_L^N +
\tilde{H}_L^N \leq \frac{1}{4} + (\frac{1}{2} + 38\delta) + 0 = \frac{3}{4} + 38\delta$. But this
function is equal $P_1 + \tilde{H}_L^N$, and so is within $4\sqrt\delta$ of some constant integer
$c$ and its integral also has to be within $4\sqrt\delta$ of $c$. As $4\sqrt\delta + 38\delta <
\frac{1}{4}$, we get $c \leq 0$, thus $\int_{Q_3} \tilde{V}_L^N \leq c + 4\sqrt{\delta} \leq
\frac{1}{4}$.

If $P_2$ is within $3\sqrt\delta$ of a constant function, and $\int_{Q_3} \tilde{G}_L^N >
\frac{1}{2} + 38\delta$, then again $P_1 + \tilde{H}_L^N$ is within $4\sqrt\delta$ from a constant
integer $c$. From lemma \ref{l3} we have in particular that $\int_{Q_3} \tilde{H}_L^N \leq
-\frac{1}{2} + 450\delta $. Obviously $\int_{Q_3} \tilde{G}_L^N \leq 1$, thus $\int_{Q_3} V_L^N =
\int_{Q_3} V_{L-1}^N + \tilde{G}_L^N + \tilde{H}_L^N \leq \frac{1}{4} + 1 - \frac{1}{2} + 450\delta
= \frac{3}{4} + 450\delta$. As this is supposed again to within $4\sqrt{\delta}$ of $c$, we have $c
\leq 0$ as $450\delta + 4\sqrt\delta \leq \frac{1}{4}$. Again thus $\int_{Q_3} \tilde{V}_L^N \leq c
+ 4\sqrt\delta \leq \frac{1}{4}$.

In the third case we suppose that $P_1'$, and thus also $P_1$ is within $3\sqrt\delta$ of a
constant function and $\int_{Q_3} \tilde{G}_L^N \leq \frac{1}{2} + 38\delta$. As $\int_{Q_3}
\tilde{V}_{L-1}^N \leq \frac{1}{4}$ from the inductive assumption, we have $\int_{Q_3} P_1 \leq
\frac{3}{4} + 38\delta$. As $P_1$ is supposed to be within $3\sqrt\delta$ of some constant integer
$c$, its integral also has to be within $3\sqrt\delta$ of $c$, which again implies $c \leq 0$ and
$\int_{Q_3} P_1 \leq  3\sqrt\delta$. As $\tilde{V}_L^N = P_1 + \tilde{H}_L^N$ and $\tilde{H}_L^N
\leq 0$, we get $\int_{Q_3} \tilde{V}_L^N \leq 3\sqrt\delta \leq \frac{1}{4}$.

The last case is when $P_1$ is within $3\sqrt\delta$ of a constant integer $c$ and $\int_{Q_3}
\tilde{G}_L^N > \frac{1}{2} + 38\delta$. In this case from lemma \ref{l3} we know there exists a
set $P' \subset Q_3$ dependent only on $v_L$ such that $|P'| = \frac{1}{2}$ and $\int_{P'}
\tilde{H}_L^N \leq -\frac{1}{2} + 450\delta$. If $P_1$ is within $3\sqrt\delta$ of a constant
integer function and $P_1 + P_2 + \tilde{H}_L^N$ is within $\delta$ of $1$ (from \ref{limit}) then
$P_2 + \tilde{H}_L^N$ is within $3\sqrt\delta + \delta \leq 4\sqrt\delta$ of some constant integer
function $C$. Taking $P_2' = P_2 - C$ we arrrive in the situation of lemma \ref{l1}:
$\tilde{H}_L^N$ depends on $v_L$ and $v_{L+1}$ while $P_2'$ depends on $v_{L+1}, v_{L+2}, \ldots,
v_Z$. This means that each of them is within $8\sqrt\delta$ of some integer function $P_3$
dependent only on $v_{L+1}$. As $\int_{P'}\tilde{H}_L^N \leq -\frac{1}{2} + 450\delta$ and
$\|\tilde{H}_L^N - P_3\| \leq 8\sqrt\delta$, we gather that $\int_{P'} P_3 \leq -\frac{1}{2} +
450\delta + 8\sqrt\delta \leq -\frac{1}{2} + 458\sqrt\delta$. As $P'$ depends only on $v_L$ and
$P_3$ only on $v_{L+1}$ and $|P'| = |Q_3 \setminus P'|$ , $$\int_{Q_3} P_3 = \int_{P'} P_3 +
\int_{Q_3 \setminus P'} P_3 = 2 \int_{P'} P_3 \leq -1  +916\sqrt\delta.$$ Returning to
$\tilde{H}_L^N$ we get $\int_{Q_3} \tilde{H}_L^N \leq \int_{Q_3} P_3 + 8\sqrt\delta \leq -1 +
924\sqrt\delta$.

As $\int_{Q_3} \tilde{G}_L^N \leq 1$ and $\int_{Q_3} \tilde{V}_{L-1}^N \leq \frac{1}{4}$ we get
$\int_{Q_3} P_1 \leq \frac{5}{4}$. As before, $\int_{Q_3} P_1$ has to be within $3\sqrt\delta$ of
the integer $c$, implying $c \leq 1$ and $\int_{Q_3} P_1 \leq 1 + 3\sqrt\delta$. We have
$\int_{Q_3} \tilde{V}_L^N = \int_{Q_3} P_1 + \tilde{H}_L^N \leq 1 + 3\sqrt\delta - 1 +
924\sqrt\delta \leq 927\sqrt\delta \leq \frac{1}{4}$.

Thus in all four cases we have completed the induction step, which proves in a finite number of
steps that $\int_{Q_3}\tilde{D}_N \leq \frac{1}{4}$. This holds for an arbitrary $N > N_0$, and
would thus have to hold for the limit function, $\int_{Q_3} d_\infty \leq \frac{1}{4}$, which
obviously contradicts the assumption that $d_\infty |_{Q_3} = \1$.\end{proof}

\begin{cor}\label{MainCor} A 3-Kadets series has a 3-point sum range, consisting of the functions
$(\0,\0)$, $(\1,\0)$ and $(\1,\1)$. As previously, this holds for any $L_p$ with $1 \leq p <
\infty$\end{cor}

\section{More points}

From the previous section we know how to make 3 points out of 2. The same mechanism can be applied
to make $r+1$ points out of $r$.

\begin{thm} For any $r > 1$ there exist a family $d_k$ of functions defined on a union of cubes
$Q_1,\ldots,Q_N$ with an $r$-point sum range. Additionally we can distinguish two disjoint subsets
$\FF$ and $\GG$ of $\{d_k : k \in \N\}$ which form a Kadets family on $Q_N$, while all other
functions $d_k$ disappear on $Q_N$. Moreover one function in the sum range of $d_k$ is equal to
$\1$ on $Q_N$ and all the other functions from the sum range disappear on $Q_N$. Finally there
exist rearrangements convergent to any point of the sum range in which the sets $\FF$ and $\GG$ are
arranged as in Proposition \ref{propsums}. \label{npoi} \end{thm}

\begin{proof} We shall prove the thesis by induction upon $r$. For $r=2$ the original Kadets
example with $N=1$ satisfies the given conditions.

Suppose we have an appropriate family for $r-1$. We add two cubes to the domain of $d_k$: $Q_{N+1}$
and $Q_{N+2}$. Denote by $x=(x_1,x_2,\ldots)$ the variable on $Q_{N+1}$ and by $y=(y_1,y_2,\ldots)$
the variable on $Q_{N+2}$. All the functions except $\GG$ will disappear on these cubes. For each
$n$ we divide the unit interval $[0,1]$ into $|M_n|$ sets $S_m^n, m \in M_n$ of measure
$\frac{1}{|M_n|}$ each. We define $g_{m,j}^n$ to be equal $\frac{1}{|J_n|}$ if $x_n \in S_m^n$, 0
otherwise. Next we define $K_n = M_{n+1}\times J_{n+1}$ and divide the unit interval $[0,1]$ into
$|K_n|$ sets $T_k^n$ of equal measure, and on $Q_{N+2}$ define $g_{m,j}^n$ to be equal to 1 if $y_n
\in T_{(m,j)}^{n-1}$, 0 otherwise. Finally to the functions $d_k$ we add a set of functions $\HH =
\{h_{m,j,k}^n\}$ which disappear on the cubes $Q_1$ to $Q_N$, and satisfy $h_{m,j,k}^n =
-\frac{1}{|K_n|}g_{m,j}^n$ on $Q_{N+1}$ and $h_{m,j,k}^n = -g_{m,j}^n \cdot g_k^{n+1}$ on
$Q_{M+2}$.

It is again easy, although tedious, to check that $\FF$, $\GG$ and the new functions $\HH$ form a
3-Kadets family on $Q_N, Q_{N+1}, Q_{N+2}$. We claim that the set $\{d_k\} \cup \HH$ satisfies the
conditions given in the theorem. The sets $\GG$ and $\HH$ form a Kadets family on $Q_{N+2}$, all
other functions disappear on $Q_{N+2}$. We have to check the sum ranges. Let us fix any convergent
rearrangement $e_k$ of $\{d_k\} \cup \HH$. From the properties of 3-Kadets families given in
section \ref{3p} we know that the limit on $Q_{N+1}$ and $Q_{N+2}$ is going to be the same, and
equal either $\0$ or $\1$. From theorem \ref{Main} we know that if the series converges to $\0$ on
$Q_M$, it has to converge to $\0$ on $Q_{N+1}$ and $Q_{N+2}$. Thus at most $r+1$ limits can be
achieved - the functions with $\0$ on $Q_N$ generate one each (by the 0-extension onto $Q_{N+1}
\cup Q_{N+2}$), while the single function with $\1$ on $Q_N$ can be extended by either $\0$ or $\1$
to $Q_{N+1} \cup Q_{N+2}$. This also satisfies the condition that only one of the points in the sum
range is $\1$ on $Q_{N+2}$, while the other points disappear on $Q_2$.

We can of course attain all the desired points in the sum range with $\GG$ and $\HH$ ordered as in
Proposition \ref{propsums} by taking the rearrangements with $\FF$ and $\GG$ ordered as in the
proposition and inserting $\HH$ as in section \ref{3p}. \end{proof}

Thus it is possible to attain a affine-independent finite set of any size $r$ as a sum range of a
conditionally convergent series. Again, this works for any $L_p$, $1\leq p <\infty$.

To attain full generality on $L_p$ we would attain arbitrary sum ranges, and not only the affine-independent sum range given above. We will do that according to the scheme from \cite{russ}, as follows:

\begin{lemma} Let $\Omega$ be an arbitrary probability space, $c_n \in \R$, $c_n \ra 0$ and let $f_n \in L_2(\Omega)$ be a sequence of integer-valued functions. Then the series $\sum_{n=1}^\infty (f_n + c_n)$ converges if and only if both $\sum_{n=1}^\infty f_n$ and $\sum_{n=1}^\infty c_n$ converge.\end{lemma}

\begin{proof} The ``if'' part is obvious. For the ``only if'' part it is enough to prove that if $\sum c_n$ diverges, then $\sum (f_n + c_n)$ has to diverge as well. In fact if $\sum c_n$ diverges then there exists an $\eps \in (0,1 \slash 4)$ such that for any $N \in \N$ we have a large Cauchy sum above $N$, i.e. for some $l > k > N$ we have $|\sum_{n=k}^l c_n| > \eps$. As $c_n \ra 0$ we can take $N$ large enough to ensure $|c_j| < \eps$ for $j > N$. Thus we can select $l = l(k)$ such that $\eps < \sum_{n=k}^{l(k)} c_n < 2\eps < \frac{1}{2}$. But then $\| \sum_{n=k}^{l(k)} (f_n + c_n)\| \geq \eps$ as a sum of an integer-valued function and a constant $c \in (\eps, 1\slash 2)$, which ensures the divergence of $\sum (f_n + c_n)$.\end{proof}

Now let us apply this lemma to our example from Theorem \ref{npoi}. We have a series $d_k$ with an $r+2$-point sum range $D$ defined on $\Omega = \bigcup_{i=1}^{2r+1} Q_i$ of cubes. We consider it as a series defined on $L_2 (\Omega)$. Let us denote $X = \lin \{\chi_{Q_1},\chi_{Q_2},\ldots,\chi_{Q_{2r+1}}\}$, i.e. the subspace of the piece-wise constant functions on $\Omega$. Let $P : L_2(\Omega) \ra X$ be the orthogonal projection onto $X$. Denote by $Y$ the subspace of $X$ consisting of those piecewise constant functions $(f_i)_{i=1}^{2r+1}$, where $f_i$ is the value of $f$ on $Q_i$, that $f_{2j} = f_{2j + 1}$ for $j = 1,2,\ldots, r$.

Recall that $\int_{Q_2j} d_k d\mi = \int_{Q_{2j+1}} d_k d\mi$ for $j=1,2,\ldots,r$. Thus for any $d_k$ we have $P(d_k) \in Y$, and thus $P(D)$ is in fact a subset of $Y$. Recall also that for odd indices $j$ the functions $d_k$ are integer-valued. Let $T : Y \ra Y$ be an arbitrary linear operator. Put $d_k' = d_k + TP(d_k)$.

\begin{thm} The sum range $D'$ of the series $\sum d_k'$ equal $(I + T)(D)$. \end{thm}

\begin{proof} The inclusion $(I + T)(D) \subset D'$ is evident. To prove the inverse inclusion consider an arbitrary arrangement $(b_k')$ of $(d_k')$ and the corresponding rearrangement $(b_k)$ of $(d_k)$. If $(b_k')$ converges to some point $b' \in D'$, then its restrictions to $Q_j$ for odd indices $j$ satisfy the conditions of the lemma. Thus the restrictions to $Q_j$ for odd $j$ of $TP(b_k)$ converge. Now the restrictions of $TP(b_k)$ to $Q_{j-1}$ are equal to the corresponding restrictions to $Q_j$, so the whole series $TP(b_k)$ converges. Then $\sum b_k = \sum (b_k' - TP(b_k))$ also has to converge. The sum of this series $b$ belongs to $D$, hence $b' = b + TP(b)$ belongs to $(I + T)(D)$.\end{proof}

This example can be transferred to any infinite-dimensional Banach space $Y$ using the results of
V.M. Kadets. In \cite{Kadets}, Theorem 7.2.2 states: {\em Let $X$ and $Y$ be Banach spaces, $X
\stackrel{f}{\Rightarrow}Y$. Suppose that $X$ has a basis $\{e_k\}_{k=1}^\infty$ and let
$\sum_{k=1}^\infty x_k$ be a series in $X$ such that SR$(\sum_{k=1}^\infty x_k)$ is not a linear
set. Then for any monotone sequence of positive numbers $\{a_k\}_{k=1}^\infty$ with $a_k \ra
\infty, k\ra \infty$, there exists a series $\sum_{k=1}^\infty y_k$ in $Y$ such that
SR$(\sum_{k=1}^\infty y_k)$ is not a linear set and $\|y_k\| \leq a_k \|x_k\|$ for all $k \in \N$},
Corollary 7.2.1 points out that if $X$ is $l_2$ then by Dvoretzky's theorem
$X\stackrel{f}{\Rightarrow}Y$, and Corollary 7.2.2 states that {\em In any infinite-dimensional
Banach space there are series whose sum range consists of two points}. This is achieved by applying
the two-point example in $L_2$ to Corollary 7.2.1 and following the proof of Theorem 7.2.2 to see
that no new points appear and all the old ones are transferred to the space $Y$. We have an
$n$-point example in $L_2$ which can be in the same manner, through obvious modifications in the
proof of Theorem 7.2.2 transferred to any Banach space $Y$. Finally for any finite-dimensional
subspaces $H_1, H_2$ of a infinitely dimensional Banach space $Y$ and any isomorphism $f:H_1 \ra
H_2$ there exists an isomorphism $\tilde{f}: Y \ra Y$ extending $f$. Thus having any $n$
points satisfying some linear equations as a sum range of $y_k$ in $Y$ we can take an $f$ transferring them to
any other $n$ points satisfying the same linear equations and then transfer the whole series by $\tilde{f}$.

\end{document}